\UseRawInputEncoding
\documentclass[12pt]{article}

\usepackage{dsfont}
\usepackage{amsfonts,amsmath,amsthm}
\usepackage{algorithm, algorithmic}
\usepackage{color}
\usepackage{graphicx}
\usepackage{stmaryrd}        

\usepackage{caption}
\usepackage{subcaption}

\usepackage[paper=a4paper,dvips,top=2cm,left=2cm,right=2cm,
foot=1cm,bottom=4cm]{geometry}

\newcommand{\dc}[1]{#1}

\newcommand{\drset}{\hat{\mathbb R}}
\newcommand{\qset}{\mathbb H}
\newcommand{\dqset}{\hat{\mathbb H}}

\newcommand{\Mdet}{\mathrm{Mdet}}
\newcommand{\Cdet}{\mathrm{Cdet}}
\newcommand{\Krdet}{\mathrm{Krdet}}
\newcommand{\Kcdet}{\mathrm{Kcdet}}
\newcommand{\qdet}{\mathrm{det}_q}

\begin{document}
	\large
	
	\title{{Moore Determinant of Dual Quaternion  Hermitian Matrices}}
	\author{Chunfeng Cui\footnote{LMIB of the Ministry of Education, School of Mathematical Sciences, Beihang University, Beijing 100191, China.
			({\tt chunfengcui@buaa.edu.cn}). }
		\and \
		Liqun Qi\footnote{Department of Applied Mathematics, The Hong Kong Polytechnic University, Hung Hom, Kowloon, Hong Kong; Department of Mathematics, School of Science, Hangzhou Dianzi University, Hangzhou 310018, China.
			({\tt maqilq@polyu.edu.hk}).}
		\and  \
		Guangjing Song \footnote{ School of Mathematics and Information Sciences, Weifang University, Weifang
			261061, China.
			({\tt sgjshu@163.com}).}
		\and and \
		Qingwen  Wang\footnote{Department of Mathematics, Shanghai University, Shanghai 200444, China.
			({\tt wqw@shu.edu.cn}).}
	}
	\date{\today}
	\maketitle

	\begin{abstract}
		
		In this paper, we extend the {Chen and Moore determinants} of quaternion {Hermitian} matrices to dual quaternion  Hermitian matrices.
		We  show  the Chen determinant of dual quaternion Hermitian {matrices} is invariant under  addition, switching,  multiplication, and unitary operations at the both hand sides.  We then show the Chen and Moore determinants {of dual quaternion Hermitian matrices} are equal to each other, and they are also equal to the products of eigenvalues.
		The characteristic polynomial of  a dual quaternion  Hermitian matrix  is also studied.
		\medskip


		
	\end{abstract}

	\renewcommand{\Re}{\mathds{R}}
	\newcommand{\rank}{\mathrm{rank}}
	\renewcommand{\span}{\mathrm{span}}
	\newcommand{\X}{\mathcal{X}}
	\newcommand{\A}{\mathcal{A}}
	\newcommand{\I}{\mathcal{I}}
	\newcommand{\B}{\mathcal{B}}
	\newcommand{\C}{\mathcal{C}}
	\newcommand{\OO}{\mathcal{O}}
	\newcommand{\e}{\mathbf{e}}
	\newcommand{\0}{\mathbf{0}}
	\newcommand{\dd}{\mathbf{d}}
	\newcommand{\ii}{\mathbf{i}}
	\newcommand{\jj}{\mathbf{j}}
	\newcommand{\kk}{\mathbf{k}}
	\newcommand{\va}{\mathbf{a}}
	\newcommand{\vb}{\mathbf{b}}
	\newcommand{\vc}{\mathbf{c}}
	\newcommand{\vq}{\mathbf{q}}
	\newcommand{\vg}{\mathbf{g}}
	\newcommand{\pr}{\vec{r}}
	\newcommand{\pc}{\vec{c}}
	\newcommand{\ps}{\vec{s}}
	\newcommand{\pt}{\vec{t}}
	\newcommand{\pu}{\vec{u}}
	\newcommand{\pv}{\vec{v}}
	\newcommand{\pn}{\vec{n}}
	\newcommand{\pp}{\vec{p}}
	\newcommand{\pq}{\vec{q}}
	\newcommand{\pl}{\vec{l}}
	\newcommand{\vt}{\rm{vec}}
	\newcommand{\vx}{\mathbf{x}}
	\newcommand{\vy}{\mathbf{y}}
	\newcommand{\vu}{\mathbf{u}}
	\newcommand{\vv}{\mathbf{v}}
	\newcommand{\y}{\mathbf{y}}
	\newcommand{\vz}{\mathbf{z}}
	\newcommand{\T}{\top}
	
	\newtheorem{Thm}{Theorem}[section]
	\newtheorem{Def}[Thm]{Definition}
	\newtheorem{Ass}[Thm]{Assumption}
	\newtheorem{Lem}[Thm]{Lemma}
	\newtheorem{Prop}[Thm]{Proposition}
	\newtheorem{Cor}[Thm]{Corollary}
	\newtheorem{example}[Thm]{Example}
	\newtheorem{remark}[Thm]{Remark}
	
	\section{Introduction}
	
	{Quaternions are extensions of complex numbers, introduced by the Irish mathematician Hamilton in 1843.    A quaternion has three imaginary parts and the quaternion multiplication is noncommutative.   Consequently, the determinant of quaternion matrices    is much more difficult than their real and complex counterparts since the quaternion numbers.
		Over the years,  many determinants of quaternion matrices with different multiplication orders are  proposed,  such as  the}  {Cayley} determinant, the  {Study} determinant, the Moore determinant {\cite{Moore22}}, the  {determinant and double determinant by Chen \cite{Ch91a, Ch91b}},  the  column and row determinants  by Kyrchei {\cite{Kyr08, Kyr12}}, etc. See {\cite{Asl96,Dy72, Li02, LSZ20}} for more details.
	
	In 1922, Eliakim Hastings Moore   defined the determinant  of quaternion  {Hermitian} matrices    \cite{Moore22}, which is widely known as   the  Moore determinant. The Moore determinant retains some basic properties of determinants such as $\Mdet(A) =0$ if and only if $A$ is singular.
	For more results on the Moore determinant, please refer to \cite{Asl96, Kyr12}.
	{However, Moore determinant is not appliable for arbitrary non-Hermitian quaternion matrices. In 1991, Chen proposed the determinant and double determinant  \cite{Ch91a, Ch91b}. In 2008 and 2012,  Kyrchei proposed  the  column and row determinants  \cite{Kyr08, Kyr12}.}

	It was  {the} British mathematician Clifford who introduced dual numbers and dual quaternions in 1873.  A {\bf dual quaternion}  has two quaternions, which are  the standard and dual parts of this dual quaternion, respectively.
	Recently, eigenvalues of dual quaternion Hermitian matrices were applied to {dual unit gain graphs \cite{CLQW24}}  {and} multi-agent formation control \cite{QC24}.    The  {latter} problem is an important application {area} {in robotics and control} \cite{QWL23, WYL12}.   To study more about the eigenvalues of dual quaternion Hermitian matrices, we need to determine the coefficients of the characteristic polynomial of a dual quaternion Hermitian matrix.  One possible way to achieve this is to apply the Moore determinant to dual quaternion Hermitian matrices.
	In this paper,  {we  study} the Moore determinant of dual quaternion Hermitian matrices.
	
	In   the next section, we review some preliminary knowledge on dual quaternion numbers, the Moore determinant and its extension, the Chen determinant.   In Section 3, we show that the Chen determinant of a dual quaternion Hermitian matrix {is invariant under the addition, switching,  multiplication, and unitary operations at the both hand sides, which is also} equal to the product of eigenvalues of that matrix.     Then in Section
	4, we show that the Moore determinant and the Chen determinant are equal for a dual quaternion Hermitian matrix.  Thus,  the Moore determinant of a dual quaternion Hermitian matrix is also equal to the product of eigenvalues of that matrix.   Finally, we study the characteristic polynomial of a dual quaternion Hermitian matrix in Section 5.

	\section{Preliminary}
	{\subsection{Quaternions and dual quaternions}}
	A {\bf quaternion} number $q\in\qset$ can be  written as  {$q=q_0+q_1\ii+q_2\jj+q_3\kk$.
		Here, $\ii,\jj,\kk$ are imaginary units that satisfy $\ii^2=\jj^2=\kk^2=\ii\jj\kk=-1$ and $\ii\jj = -\jj\ii = \kk$.  Thus, the multiplication of quaternion numbers is not commutative.
		The conjugate of $q$ is  $q^*=q_0-q_1\ii-q_2\jj-q_3\kk$.   We have 
		$|q| = \sqrt{q_0^2 + q_1^2 + q_2^2 + q_3^2}$.}
	
	A {\bf dual quaternion} number $q\in\dqset$ can be  written as $q=q_s+q_d\epsilon$, where
	$q_s, q_d \in \qset$, $q_s$ is the standard part of $q$, $q_d$ is the dual part of $q$, $\epsilon$ is the infinitesimal unit, $\epsilon \not = 0$, $\epsilon^2=0$, $\epsilon$ is commutative with quaternion numbers.  If $q_s \not = 0$, then $q$ is called {\bf appreciable}.
	The conjugate of $q$ is  $q^*=q_s^*+q_d^*\epsilon$.  We have 
	\begin{equation} \label{e1}
		|p| := \left\{ \begin{aligned} |p_s| + {(p_s^*p_d+p_d^*p_s) \over 2|p_s|}\epsilon, & \ \  {\rm if}\  p_s \not = 0, \\
			|p_d|\epsilon, &  \ \  {\rm otherwise}.
		\end{aligned} \right.
	\end{equation}
	
	{If both $q_s$ and $q_d$ are real numbers, then $q = q_s + q_d\epsilon$ is called a {\bf dual number}.}  

	\begin{Prop}\label{prop:prelim_dq}
		Let $\dc q_1$ and  $\dc q_2 \in \dqset$ be two dual  quaternion numbers. Then
		\begin{itemize}
			\item[(i)] $Re(\dc q_1)\le |\dc q_1|$ and the equality holds if and only if $\dc q_1$ is a nonnegative dual number.
			\item[(i)] $Re(\dc q_1) =Re({\dc q^*_1})$ and $Re(\dc q_1\dc q_2) =Re(\dc q_2\dc q_1)$.
			{\item[(ii)] $q_1+q_1^*\in \drset$.}
			\item[(iii)] $|\dc q_1\dc q_2| = |\dc q_1| |\dc q_2|$.
			\item[(iv)] $|\dc q_1+\dc q_2| \le |\dc q_1| +|\dc q_2|$.
		\end{itemize}
	\end{Prop}

	A dual quaternion vector $\vx\in\dqset^{n}$ can be denoted as $\vx=\vx_s+\vx_d\epsilon$. If $\vx_s\neq \0$, we say $\vx$ is appreciable. Otherwise, it is  infinitesimal.
	If $\vx$ is appreciable, then the 2-norm of $\vx$ is defined by  $\sqrt{\sum_{i=1}^n|x_i|^2}$. Otherwise, if $\vx$ is  infinitesimal, then $\|\vx\|=\|\vx_d\|\epsilon$.

	The collection of $m \times n$ dual quaternion matrices is denoted by  $\dqset^{m \times n}$.
	A dual quaternion matrix $A= (a_{ij}) \in \dqset^{m \times n}$ can be denoted as $	A= A_s + A_d\epsilon$,
	where $A_s, A_d \in \qset^{m \times n}$.   The conjugate transpose of $A$ is $A^* = (a_{ji}^*)$.
	Let $A\in \dqset^{m \times n}$ and $B \in \dqset^{n \times r}$.   Then we have $(AB)^* = B^*A^*$.
	
	Given a square dual quaternion matrix $A\in \dqset^{n \times n}$, it is called invertible (nonsingular) if $AB = BA= I_n$ for some $B \in \dqset^{n \times n}$, where $I_n$ is the $n\times n$ identity matrix. Such $B$ is unique and denoted by  $A^{-1}$.
	Square dual quaternion matrix $A$ is called Hermitian if $A^* = A$. Then $A$ is Hermitian if and only if both $A_s$ and $A_d$ are quaternion Hermitian matrices.  Square dual quaternion matrix $A$ is called unitary if $A^*A= I_n$. Apparently,  $A\in \dqset^{n \times n}$ is unitary if and only if its column vectors form an orthonormal basis of $\dqset^n$.
	
	An $n \times n$ dual quaternion Hermitian matrix has exactly $n$ right eigenvalues, which are all dual numbers and also left eigenvalues \cite{QL23}.   We simply call them the eigenvalues of that matrix.

	\subsection{Quaternion determinants}
	
	{A {\bf permutation} $\sigma=\{i_1,\dots,i_n\}\in S_n$ denotes a function $\sigma$ such that $\sigma(j) = i_j$ for all $j=1,\dots,n$. We could also rewrite it as a two-line formulation
		\[\begin{pmatrix}
			1& 2 & \cdots & n-1 & n\\
			i_1 & i_2 & \cdots & i_{n-1} & i_n
		\end{pmatrix}.\]
		A {\bf permutation cycle} $\sigma=(i_1,\dots,i_k)$  is a subset of a permutation whose elements trade places with one another.  Namely, $\sigma(i_j) = i_{j+1}$ for $j=1,\dots,k-1$ and $\sigma(i_k)=i_1$.
		A cycle of length $k$ is called a {\bf $k$-cycle.}
		We could also rewrite the permutation cycle as a two-line formulation
		\begin{equation}\label{equ:2line_cycle}
			\begin{pmatrix}
				i_1 & i_2 & \cdots & i_{k-1} & i_k \\
				i_2 & i_3 & \cdots & i_k & i_1
			\end{pmatrix}.
		\end{equation}
		For instance, the permutation cycle $\sigma=(312)$ represents a permutation $\sigma(3)=1$, $\sigma(1)=2$, and $\sigma(2)=3$.
		The sign of any $k$-cycle is $(-1)^{k-1}$. In fact, by implementing $k-1$ switches, i.e.,  $i_1$ with $i_{k}$,  $i_1$ with $i_{k-1}$,  $\dots$,   $i_1$ with $i_{2}$,  sequentially,  \eqref{equ:2line_cycle} reduces to $$ \begin{pmatrix}
			i_1 & i_2 & \cdots & i_{k-1} & i_k \\
			i_1 & i_2 & \cdots & i_{k-1} & i_k
		\end{pmatrix},$$
		which is corresponding to $\{1,\dots,k\}$.

		The {\bf cyclic decomposition} of a permutation  is to decompose the permutation as a product of disjoint cycles,
		\begin{equation}\label{def:cyc_decomp}
			\sigma=(n_{11}\cdots n_{1l_1})(n_{21}\cdots n_{2l_2})\cdots(n_{r1}\cdots n_{rl_r}).
		\end{equation}
		Every permutation can be written as  a product of disjoint cycles. For instance,
		$\{1,2,3\} = (1)(2)(3)$ and $\{2,1,3\} = (12)(3)$.
		The {\bf inverse  permutation}   $\sigma^{-1}$  is the inverse  of   $\sigma$ such that $\sigma(\sigma^{-1}(i)) = i$ and $\sigma^{-1}(\sigma(i)) = i$ for $i=1,\dots,n$. In the cycle form,  the inverse  permutation just reverses the direction of each cycle. Namely,  $\sigma^{-1}=(i_1,i_n,i_{n-1},\dots,i_2)$.
		For more details on the permutation, we refer to \cite{Ca99}.
	}

	Let $A=(a_{ij})$ be a quaternion Hermitian matrix in $\qset^{n\times n}$ and $\sigma$ be a permutation of $S_n=\{1,\dots,n\}$.
	In 1922,  Eliakim Hastings Moore     defined the {\bf Moore determinant} for quaternion matrices    \cite{Moore22}   as follows,
	\begin{equation}\label{equ:Mdet}
		\Mdet(A) = \sum_{\sigma\in {S_n^M}} {s(\sigma)}a_{\sigma},
	\end{equation}
	where  {the permutation cycle} $\sigma$ is a product of disjoint cycles  {as in \eqref{def:cyc_decomp} and} satisfies
	\begin{equation}\label{equ:sigma_Mdet}
		\begin{aligned}
			n_{i1}<n_{ij} &\text{ for all } j>1, i=1,\dots,r, \text{ and } n_{11}>n_{21}>\cdots>n_{r1},
		\end{aligned}
	\end{equation}
	${s(\sigma)}$ denotes the {sign} of $\sigma$, and
	\[a_{\sigma} = (a_{n_{11},n_{12}}a_{n_{12},n_{13}}\cdots a_{n_{1l_1},n_{11}})(a_{n_{21},n_{22}}\cdots a_{n_{2l_2},n_{21}})\cdots(\cdots a_{n_{rl_r},n_{r1}}).\]
	{Denote}  $S_n^M$ as the set of permutations satisfying \eqref{equ:sigma_Mdet}.
	When the matrix is a complex or real matrix, the Moore determinant \eqref{equ:Mdet} reduces to the ordinary determinant.
	Following \cite{Ch91a}, we also denote  $\sigma_i=(n_{i1}\cdots n_{i l_i})$,
	\[ \langle\sigma_i \rangle=\langle n_{i1}\cdots n_{i l_i}n_{i1}\rangle = a_{n_{i1},n_{i2}}a_{n_{i2},n_{i3}}\cdots a_{n_{il_i},n_{i1}}.\]
	Thus,  $\sigma=\sigma_1\cdots\sigma_r$ and $ \langle\sigma  \rangle = \langle\sigma_1\rangle \cdots \langle\sigma_r\rangle$.

	Let $\va_{\cdot j}$ be the $j$th column and $\va_{i\cdot}$ be the $i$th row of a matrix $A\in\qset^{n\times n}$ or $A\in\dqset^{n\times n}$.
	Let $A_{\cdot j}(\va)$ be a matrix obtained  {from} replacing the $j$th column of $A$ by the column vector  $\va$, and $A_{i\cdot}(\vb)$ be a matrix result from replacing the $i$th row of $A$ by the row vector  $\vb$. Denote by $A^{ij}$ a submatrix of $A$ obtained by deleting both the $i$th row and  the $j$th column.
	Here,  $A_{\cdot j}(\va)$ and $A_{i\cdot}(\vb)$ may not be Hermitian matrices {even if $A$ is Hermitian}.
	Such matrices are {\bf almost Hermitian} \cite{Dy72}.
	{We say a matrix is $k$-almost Hermitian if it is self-adjoint except for the $k$-th row or  column.}

	In 1972, Dyson \cite{Dy72} presented an equivalent formulation of  {the} Moore determinant  for $k$-almost Hermitian matrices as follows ,
	\begin{equation}\label{equ:Mdet2}
		\Mdet(A) = a_{kk} \Mdet(A^{kk}) - \sum_{i=1,i\neq k}^{n}a_{ki}\Mdet(A^{kk}_{\cdot i}(a_{\cdot k})).
	\end{equation}
	Here, $k$ is the index in  the $k$-almost Hermitian.
	If $A$ is Hermitian, then the above value remains the same for all $k=1,\dots,n$.
	{This formulation is quite useful for the numerical  {computation}.}

	However, the Moore determinant is only  applicable to Hermitian or almost Hermitian matrices.
	Several  researchers  have  considered determinants of  arbitrary  quaternion matrices.
	In 1922, Study \cite{St22} proposed determinants of quaternion matrices through their injective algebra homomorphism with complex matrices. Specifically, let $A=A_1+A_2\jj$. Then
	\[\mathrm{Sdet}(A) = \mathrm{det}\left(
	\begin{bmatrix}
		A_1 & -\bar A_2 \\
		A_2 & \bar A_1
	\end{bmatrix}
	\right).\]	In 1991, Chen \cite{Ch91a, Ch91b} proposed a novel definition of quaternion determinant as follows
	\begin{equation}\label{equ:Cdet}
		\Cdet(A) = \sum_{\sigma\in {S_n^C}} (-1)^{n-r} a_{\sigma},
	\end{equation}
	where   {the permutation cycle} $\sigma$ is a product of disjoint cycles {as in \eqref{def:cyc_decomp} and} satisfies
	\begin{equation}\label{equ:sigma_Cdet}
		\begin{aligned}
			n_{i1}>n_{ij} &\text{ for all } j>1, i=1,\dots,r, \text{ and } n=n_{11}>n_{21}>\cdots>n_{r1},
		\end{aligned}
	\end{equation}
	and  $S_n^C$ denotes the set of permutations satisfying \eqref{equ:sigma_Cdet}.
	In the following, we call it  the {\bf Chen determinant} and denote it by  {$\Cdet$}.
	In \eqref{equ:Cdet},  $n$ is fixed as the first element in the cycle $\sigma$.

	In 2008 and also in 2012, Kyrchei \cite{Kyr08, Kyr12} generalized the first element  {in the} Chen determinant to any index  $i\in\{1,\dots,n\}$ and  proposed the column and row determinants of  {arbitrary} quaternion matrices as follows
	\begin{equation}\label{equ:rdet}
		\Krdet_i(A) = \sum_{\sigma\in {S_n^K}} (-1)^{n-r} a_{\sigma},
	\end{equation}
	where $\sigma$ is the permutations of $\{1,\dots,n\}$ satisfying
	\begin{equation}\label{equ:sigma_Krdet}
		\begin{aligned}
			&\sigma=(ii_{k_1}i_{k_1+1}\cdots i_{k_1+l_1})(i_{k_2}i_{k_2+1}\cdots i_{k_2+l_2})\cdots(i_{k_r}i_{k_r+1}\cdots i_{k_r+l_r}), \\
			i_{k_2}&<i_{k_3}<\cdots <i_{k_r}\text{ and } i_{k_t}<i_{k_t+s} \text{ for all } t=2,\dots,r, s=1,\dots,l_t,
		\end{aligned}
	\end{equation}
	and  $S_n^K$ denotes the set of permutations satisfying \eqref{equ:sigma_Krdet}.
	In the following, we call it  the {\bf  Kyrchei row determinant}.
	
	Similarly, the  {\bf  Kyrchei column determinant} is defined by
	\begin{equation}\label{equ:rdet}
		\Kcdet_j(A) = \sum_{\sigma\in {\bar S_n^K}} (-1)^{n-r} a_{\sigma},
	\end{equation}
	where $\sigma$ is the permutations of $\{1,\dots,n\}$ satisfying
	\begin{equation}\label{equ:sigma_Krdet}
		\begin{aligned}
			&\sigma=(j_{k_r}j_{k_r+l_r}\cdots j_{k_r+1}j_{k_r})\cdots (j_{k_2+l_2}\cdots j_{k_2+1}j_{k_2}) (j_{k_1+l_1}\cdots j_{k_1+1}j_{k_1}j), \\
			j_{k_2}&<j_{k_3}<\cdots <j_{k_r}\text{ and } j_{k_t}<j_{k_t+s} \text{ for all } t=2,\dots,r, s=1,\dots,l_t.
		\end{aligned}
	\end{equation}
	When $A$ is a quaternion Hermitian matrix, all Kyrchei row and column determinants are real numbers and are the same, which  also coincide with the Moore determinant.
	Furthermore, the Kyrchei column and row determinants   may  derive the formulation of the inverse  matrix  by the classical adjoint matrix and could be connected  to the solution of linear systems.

	{We now review} the elementary row and column operations {for quaternion} matrices.
	{Consider} the three   elementary row and column operations that can	be used to transform a matrix (square or not) into a simple form  and can  facilitate   {determinant} calculations. In the following definition, we focus on row operations, which are implemented by left  {multiplying} the
	matrices. Column operations can be defined and used in a similar
	fashion, while the matrices that implement them act on the right.
	
	\begin{Def}[Elementary row and column operations]\label{def:elem}
		Given a  {quaternion} matrix $A\in{\qset}^{n\times n}$ and $i,j\in\{ 1,\dots,n\}$.	Consider the following three elementary transformations.

		(i)  Switching of two rows. Let $P_{ij}=(a_{kl}) \in\mathbb R^{n\times n}$ satisfy $$a_{kl}=\left\{\begin{array}{cl}
			1, & \mathrm{ if }\ k=l\neq i \ \mathrm{ and }\ k=l\neq j;\\
			1, & \mathrm{ if }\ k=i\ \mathrm{ and }\  l=j;\\
			1, & \mathrm{ if }\ k=j\ \mathrm{ and }\  l=i;\\
			0, & \mathrm{otherwise}.
		\end{array}\right.$$
		Then {by} multiplying $P_{ij}$ on the left, we switch the $i$-th and the $j$-th rows {of} $A$. We refer $P_{ij}$ as a {\bf switching matrix}.

		(ii)  Multiplication of a row by a scalar.  Let $c\in{\qset}$ be any {quaternion number} and  $P_{i;c}=(a_{kl})\in{\qset}^{n\times n}$ satisfy $${a_{kl}}=\left\{\begin{array}{cl}
			1, & \mathrm{ if }\ k=l\neq i;\\
			c, & \mathrm{ if }\ k=l=i;\\
			0, & \mathrm{otherwise}.
		\end{array}\right.$$
		Then   $P_{i;c}A$   multiplies   the $i$-th   row of  $A$ by $c$ {on} the left.  We refer $P_{i;c}$ as a {\bf  multiplication matrix}.
		
		(iii) Addition of a scalar multiple of one row to another row.   Let $c\in{\qset}$ be any {quaternion number} and  $P_{ij;c}=(a_{kl})\in\qset^{n\times n}$ satisfy  $$a_{kl}=\left\{\begin{array}{cl}
			1, & \mathrm{ if }\ k=l;\\
			c, & \mathrm{ if }\ k=j\ \mathrm{and}\ l=i;\\
			0, & \mathrm{otherwise}.
		\end{array}\right.$$
		Then   $P_{ij;c}A$   adds the $j$-th row of $A$ by  the  $i$-th   row of  $A$ multiplied {with} $c$ one the left.   We refer $P_{ij;c}$ as an {\bf addition matrix}.
	\end{Def}
	
	{For more properties and applications of quaternion determinants, people may check \cite{KTP24, Li02, LLJ22, LSZ20}.	
		
		We note} {that the Moore determinant, the Chen determinant,  the  Kyrchei determinant, and Definition~\ref{def:elem} can be extended to dual quaternions without any difficulty.}
	{However, the study of properties of dual quaternion determinants may be different.
		Very recently, Ling and Qi  \cite{LQ24} studied the determinant properties of dual complex matrices, and then introduced the concept 	of  quasi-determinant of dual quaternion matrices. Based upon these,  they showed  the quasi-determinant of a dual quaternion Hermitian matrix is equivalent to the product
		of the square of the magnitudes of all eigenvalues.
		In this paper, we focus on the determinants  of dual quaternion Hermitian matrix that are equal to the product of eigenvalues, such as the Chen determinant and the Moore determinant.
	}

	Next, we present several  {properties} of the Chen determinant {for quaternion Hermitian matrices} \cite{Ch91a, Ch91b}.
	\begin{Prop}\label{prop_Cdet0}
		Let $A=(a_{ij})$ be a quaternion Hermitian matrix in $\qset^{n\times n}$, $P_{ij}$ be a switching matrix, $P_{i;\alpha}$ be a multiplication matrix, and $P_{ij;\alpha}$ be an addition matrix.
		Then the following results hold.
		
		\begin{itemize}
			\item[(i)] $\Cdet(P_{ij}^*AP_{ij}) = \Cdet(A)$;
			
			\item[(ii)]  $\Cdet(P_{n;\alpha}A) = \alpha\Cdet(A)$, $\Cdet(AP_{n;\alpha}) = \alpha\Cdet(A)$, and $\Cdet(P_{i;\alpha}^*AP_{i;\alpha}) = \alpha^*\alpha\Cdet(A)$;
			
			\item[(iii)] $\Cdet(P_{ij;\alpha}^*AP_{ij;\alpha}) = \Cdet(A)$;
			
			\item[(iv)] For any unitary matrix $U\in\qset^{n\times n}$, we have $\Cdet(U^*AU) = \Cdet(A)$;
			
		\end{itemize}
	\end{Prop}
	
	{Recall that an $n \times n$ quaternion Hermitian matrix has exactly $n$ right eigenvalues, which are all real numbers and also left eigenvalues \cite{Zh97}.   We simply call them eigenvalues of that matrix.}
	{Recently, Qi and Luo \cite{QL23} showed  an $n \times n$ dual quaternion Hermitian matrix has exactly $n$ right eigenvalues, which are all dual real numbers and also left eigenvalues. {As mentioned early, they are also simply called eigenvalues.}
		In what follows, we show the Chen determinant of dual quaternion Hermitian  {matrices} is equal to the product of these $n$ eigenvalues.}
	
	\section{Properties of the  Dual Quaternion Chen Determinant}
	
	%
	

	{In this section, we} show that the   {results in Proposition~\ref{prop_Cdet0}} can be generalized to dual quaternion Hermitian matrices.
	We begin with the following lemma.
	\begin{Lem}\label{Lem:<s>2+=2}
		Let  $A=(a_{ij})$ be  a dual quaternion Hermitian matrix in $\dqset^{n\times n}$.
		Suppose $\sigma_0=(p_1p_2\cdots p_skq_1\cdots q_t)$ is a cycle factor of the permutation $\sigma\in S_n$. Denote
		\begin{eqnarray*}
			&\bar\sigma_0=(p_1q_t\cdots q_1 k p_s\cdots p_2),\\
			&\sigma_0^+=(kq_1\cdots q_tp_1p_2\cdots p_s),\\
			&\bar\sigma_0^+=(k p_s\cdots p_2p_1q_t\cdots q_1 ).
		\end{eqnarray*}
		Then we have
		\begin{equation}\label{sigma_bar^+_permute}
			\langle \sigma_0 \rangle + \langle \bar\sigma_0 \rangle = 	\langle \sigma_0^+ \rangle + \langle \bar\sigma_0^+ \rangle.
		\end{equation}
	\end{Lem}
	\begin{proof}
		By direct computation, we have $\langle \sigma_0 \rangle=a_{p_1p_2}\cdots a_{p_sk}a_{kq_1}\cdots a_{q_tp_1}$ and $\langle \bar\sigma_0 \rangle=a_{p_1q_t}\cdots a_{q_1k}a_{kp_s}\cdots a_{p_2p_1}$. Therefore, we have
		$\langle \sigma_0 \rangle = \langle \bar\sigma_0 \rangle ^*$. Similarly, we have
		$\langle \sigma_0^+ \rangle = \langle \bar\sigma_0^+ \rangle ^*$.
		Thus, both the right  {and  left hand sides} of \eqref{sigma_bar^+_permute} are dual numbers.
		
		Furthermore, let $w_1=a_{p_1p_2}\cdots a_{p_sk}$, $w_2=a_{kq_1}\cdots a_{q_tp_1}$. By Proposition~\ref{prop:prelim_dq}, we have $Re(\langle \sigma_0 \rangle)=Re(w_1w_2) = Re(w_2w_1)=Re(\langle\sigma_0^+\rangle)$. Similarly, there is  $Re(\langle \bar\sigma_0 \rangle)=Re(\langle\bar\sigma_0^+\rangle)$.
		Thus, the real part  of the right hand side of \eqref{sigma_bar^+_permute} is equal to that of the    left hand side.
		{Combining} with the fact that  both the right and left hand sides of  \eqref{sigma_bar^+_permute} are  {dual numbers}, we derive \eqref{sigma_bar^+_permute}.
		
		This completes the proof.
	\end{proof}
	
	\begin{Lem}\label{Lem:switch_Cdet}
		Let $A=(a_{ij})$ be  a dual quaternion Hermitian matrix in $\dqset^{n\times n}$, and $P_{ij}$ be a switching matrix.  {Then}   $\Cdet(P_{ij}^*AP_{ij}) = \Cdet(A)$.
	\end{Lem}
	\begin{proof}
		If $i,j,n$ are mutually {distinct}, then we have $P_{ij} = P_{jn}P_{in}P_{jn}$. Hence,  {it suffices} to show $\Cdet(P_{jn}^*AP_{jn}) = \Cdet(A)$ for $j\neq n$.
		For any $\sigma=\sigma_1\cdots \sigma_r\in S_n^C$,  denote {the apostrophe of $\sigma$}
		$$\sigma' = \sigma |_{n\longleftrightarrow j}$$ as  the permutation of $S_n$ derived from interchanging $n$ and $j$ in $\sigma$.
		Since the cycle structures of $\sigma$ and $\sigma'$ are the same,   they have the same parity. Thus, we have
		\[\Cdet(P_{jn}^*AP_{jn}) = \sum_{\sigma\in {S_n^C}} (-1)^{n-r} a_{\sigma'}.\]
		Consider the following two cases.
		
		Case (a). Suppose that  $n$ and $j$ are in the same cycle factor. Then there {exist} nonnegative integers $s,t$ such that  $\sigma_1=(np_1\cdots p_sjq_1\cdots q_t)$.
		Let $\bar\sigma_1=(nq_t\cdots q_1jp_s\cdots p_1)$, $\delta_1=(nq_1\cdots q_tjp_1\cdots p_s)$, $\bar\delta_1=(np_s\cdots p_1jq_t\cdots q_1)$ be the first cycle factors of $\bar\sigma=\bar\sigma_1\sigma_2\cdots \sigma_r\in S_n^C$, $\delta=\delta_1\sigma_2\cdots \sigma_r\in S_n^C$, $\bar\delta=\bar\delta_1\sigma_2\cdots \sigma_r\in S_n^C$, respectively.
		Then by Lemma~\ref{Lem:<s>2+=2}, we have  $\langle \sigma_1 \rangle + \langle \bar\sigma_1 \rangle = 	\langle \delta_1' \rangle + \langle \bar\delta_1' \rangle$ and $\langle \delta_1 \rangle + \langle \bar\delta_1 \rangle = 	\langle \sigma_1' \rangle + \langle \bar\sigma_1' \rangle$. 	
		{Here, the apostrophe of a cycle denotes the  interchanging of $n$ and $j$.}
		Thus, we have
		\[\langle \sigma \rangle + \langle \bar\sigma \rangle + \langle \delta \rangle + \langle \bar\delta \rangle = 	\langle \sigma'\rangle + \langle \bar\sigma'\rangle + 	\langle \delta' \rangle + \langle \bar\delta' \rangle.\]
		Namely, we have built {the correspondence between} the permutation terms in $S_n^C$ {with the permutation} terms of $\Cdet(P_{jn}^*AP_{jn})$.

		Case (b). Suppose that  $n$ and $j$ are in two distinct cycle factors. Without loss of generality, let
		\[\sigma=(np_1\cdots p_s) \cdots(jq_1\cdots q_t) \cdots, \ s,t\ge 0,\]
		$j>\max_{i=1}^s p_i$, and $j>\max_{i=1}^t q_i$.
		Otherwise, {if $j>\max_{i=1}^s p_i$ or $j>\max_{i=1}^t q_i$ does not hold}, we can use the similar technique in Lemma~\ref{Lem:<s>2+=2} to get the corresponding permutations without changing the summation values.
		
		Let
		\[\bar \sigma=(np_s\cdots p_1) \cdots(jq_1\cdots q_t) \cdots,\]
		\[\delta=(np_1\cdots p_s) \cdots(jq_t\cdots q_1) \cdots,\]
		\[\bar \delta=(np_s\cdots p_1) \cdots(jq_t\cdots q_1) \cdots,\]
		{and $\Sigma$, $\bar\Sigma$, $\Delta$, $\bar \Delta$ be the cycles that interchanging $p_1\cdots p_s$ with $q_1\cdots q_t$ in $\sigma$, $\bar\sigma$, $\delta$, $\bar \delta$, respectively.}
		Then we have
		\begin{eqnarray*}
			&&	\langle\sigma\rangle + \langle\bar\sigma \rangle + \langle\delta \rangle + \langle\bar\delta \rangle\\
			&=& \left[ 	\langle np_1\cdots p_sn \rangle + 	\langle np_s\cdots p_1n \rangle \right]\cdots \left[ 	\langle jq_1\cdots q_tj \rangle + 	\langle jq_t\cdots q_1i \rangle \right]\cdots\\
			&=& \left[ 	\langle jq_1\cdots q_tj \rangle + 	\langle jq_t\cdots q_1j \rangle \right]\cdots \left[ 	\langle np_1\cdots p_sn \rangle + 	\langle np_s\cdots p_1n \rangle \right]\cdots\\
			&=&	{\langle\Sigma'\rangle + \langle\bar\Sigma' \rangle + \langle\Delta' \rangle + \langle\bar\Delta' \rangle.}
		\end{eqnarray*}
		Here, the second equality follows from the fact that the numbers in the square brackets are dual numbers, which are communicative with dual quaternion numbers.
		Similarly, we have $\langle\Sigma\rangle + \langle\bar\Sigma \rangle + \langle\Delta \rangle + \langle\bar\Delta \rangle=\langle\sigma'\rangle + \langle\bar\sigma' \rangle + \langle\delta' \rangle + \langle\bar\delta' \rangle$.
		Thus, we have built the correspondence between the factors in  $\Cdet(P_{ij}^*AP_{ij})$ and $\Cdet(A)$.

		Following this scheme, we could show $\Cdet(P_{ij}^*AP_{ij}) = \Cdet(A)$.	
		This completes the proof.
	\end{proof}

	\begin{Lem}\label{Lem:rowprod}
		Let $A=(a_{ij})$ be a dual quaternion Hermitian matrix in $\dqset^{n\times n}$, and $P_{n;\alpha}$ be a multiplication matrix.  {Then  $\Cdet(P_{n;\alpha}A) = \alpha\Cdet(A)$, $\Cdet(AP_{n;\alpha}) = \alpha\Cdet(A)$, and $\Cdet(P_{i;\alpha}^*AP_{i;\alpha}) = \alpha^*\alpha\Cdet(A)$ for all $i=1,\dots,n$.}
	\end{Lem}
	\begin{proof}
		Let $\sigma=\sigma_1\cdots \sigma_r\in S_n^C$ be a permutation of $S_n$ satisfying \eqref{equ:sigma_Cdet} {and $\sigma_1=(np_1\cdots p_s)$.}
		For the $n$-th row multiplication, we have
		\begin{eqnarray*}
			\Cdet(P_{n;\alpha}A) &=& \sum_{\sigma\in S_n^C} (-1)^{n-r} \alpha a_{np_1}  a_{p_1p_2}\cdots a_{p_sn}\langle \sigma_2 \rangle\cdots \langle \sigma_r \rangle  \\
			&=&  \alpha \left(\sum_{\sigma\in {S_n^C}} (-1)^{n-r}\langle \sigma_1 \rangle\cdots \langle \sigma_r \rangle \right)  \\
			&=&  \alpha \Cdet(A).
		\end{eqnarray*}

		{In addition,}  there exists  $\bar\sigma_1=(n p_s\cdots  p_1)$ such that $\bar\sigma=\bar\sigma_1 \sigma_2\cdots \sigma_r\in S_n^C$ and $\langle \sigma_1 \rangle + \langle \bar\sigma_1\rangle \in \drset$.
		For the $n$-th column multiplication, we have
		\begin{eqnarray*}
			\Cdet(AP_{n;\alpha}) &=&\sum_{\sigma\in S_n^C} (-1)^{n-r} a_{np_1}  a_{p_1p_2}\cdots a_{p_sn} \alpha\langle \sigma_2 \rangle\cdots \langle \sigma_r \rangle  \\
			&=& \sum_{\sigma\in {S_n^C}} (-1)^{n-r}\langle \sigma_1 \rangle \alpha  \langle \sigma_2 \rangle \cdots \langle \sigma_r \rangle\\
			&=& \frac12 \sum_{\sigma\in {S_n^C}} (-1)^{n-r}  \left[  \langle \sigma_1 \rangle+\langle \bar\sigma_1 \rangle\right]\alpha\langle \sigma_2 \rangle\cdots \langle \sigma_r \rangle   \\
			&=& \frac{\alpha}2 \sum_{\sigma\in {S_n^C}} (-1)^{n-r}  \left[ \langle \sigma_1 \rangle+\langle \bar\sigma_1 \rangle\right]\langle \sigma_2 \rangle\cdots \langle \sigma_r \rangle   \\
			&=&  \alpha \Cdet(A).
		\end{eqnarray*}
		Here, the third equality is true because  we have repeated all permutations twice and the fourth equality is true since the number in the square bracket is a dual number.
		
		For the $i$-th row and column multiplications, we have
		\begin{eqnarray*}
			\Cdet(P_{i;\alpha}^*AP_{i;\alpha}) &=& \Cdet(P_{in}^*P_{i;\alpha}^*AP_{i;\alpha}P_{in})\\
			&=& \Cdet\left((P_{i;\alpha}P_{in})^*A(P_{i;\alpha}P_{in})\right)\\
			&=& \Cdet(P_{n;\alpha}^*AP_{n;\alpha})\\
			&=&  \alpha^*\alpha \Cdet(A).
		\end{eqnarray*}
		
		This completes the proof.
	\end{proof}
	
	\begin{Lem}\label{Lem:An_row_ai=0}
		Let $A=(a_{ij})$ be a dual quaternion Hermitian matrix in $\dqset^{n\times n}$.  Then
		\[\Cdet(A_{n\cdot}(a_{k\cdot}))=0,\quad \forall\, k=1,\dots,n-1.\]
		Furthermore, {we have}	$\Cdet(A_{n\cdot}(\alpha a_{k\cdot}))=0$.
		{In other words, if we replace the $n$-th row of a dual quaternion Hermitian matrix by the $k$-th row (possibly multiplied by a dual quaternion number $\alpha$ on the left), then the Chen determinant of the result matrix is zero.}
	\end{Lem}
	\begin{proof}
		{Let $s,t$ be nonnegative integers.}  We first divide all permutations in $S_n^C$ into three sets  depending on the location of $k$ as follows.
		
		(i) $S_1=\{\sigma_1(k): \sigma_1(k) = (np_1\cdots p_s) \cdots (k q_1\cdots q_t)\cdots\}$;
		
		(ii) $S_2=\{\sigma_2(k): \sigma_2(k)  = (np_1\cdots p_s) \cdots (u_1\cdots u_l k q_1\cdots q_t)\cdots\}$;
		
		(iii) $S_3=\{\sigma_3(k): \sigma_3(k)  = (np_1\cdots p_sk q_1\cdots q_t)\cdots\}$.
		
		We show this lemma by proving that for any permutation  $\sigma$ in the first two cases, there is a permutation  $\bar\sigma$ in the same  {case}  and permutations  $\sigma^+$ and  $\bar\sigma^+$ in the third case  such that 	
		{\begin{equation}\label{equ:ss=ss}
				s(\sigma)\langle \sigma \rangle + s(\bar\sigma)\langle  \bar\sigma \rangle= -s(\sigma^+)\langle \sigma^+ \rangle -s(\bar\sigma^+)\langle \bar\sigma^+ \rangle.
		\end{equation}}
		
		Case (i). Let $\sigma\in S_1$ and
		$$\bar\sigma = (np_1\cdots p_s) \cdots (k q_t\cdots q_1)\cdots,$$
		$$\sigma^+ = (n q_1\cdots q_t kp_1\cdots p_s) \cdots, \ \ \bar\sigma^+ = (nq_t\cdots q_1kp_1\cdots p_s)  \cdots.$$
		Then we have
		\begin{eqnarray*}
			\langle \sigma \rangle + \langle  \bar\sigma \rangle &=&  \langle np_1\cdots p_s n\rangle \cdots  \langle k q_1\cdots q_t k\rangle \cdots + \langle np_1\cdots p_s n\rangle \cdots  \langle k q_t\cdots q_1 k\rangle \cdots\\
			&=&  \langle kp_1\cdots p_s n\rangle \cdots  \langle k q_1\cdots q_t k\rangle \cdots + \langle kp_1\cdots p_s n\rangle \cdots  \langle k q_t\cdots q_1 k\rangle\cdots \\
			&=&  \langle k q_1\cdots q_t kp_1\cdots p_s n\rangle \cdots   + \langle k q_t\cdots q_1 kp_1\cdots p_s n\rangle \cdots  \\
			&=&  \langle n q_1\cdots q_t kp_1\cdots p_s n\rangle \cdots   + \langle n q_t\cdots q_1 kp_1\cdots p_s n\rangle \cdots\\
			&=& \langle \sigma^+\rangle + \langle  \bar\sigma^+ \rangle ,
		\end{eqnarray*}
		where the second and the {fourth} equalities follow from the fact that the last row of $A$ is {replaced by} the $k$-th row, and the third equality follows from the fact that $ \langle k q_1\cdots q_t k\rangle+ \langle k q_t\cdots q_1 k\rangle$ is a dual number, {which} is communicative  with  dual quaternion numbers {and can be moved    to the beginning of the products.}
		Furthermore, suppose $\sigma$ has $r$ distinct cycles.   Then $s(\sigma)=s(\bar\sigma)=(-1)^{n-r}$ and $s(\sigma^+)=s(\bar\sigma^+)=(-1)^{n-r+1}$. Thus, \eqref{equ:ss=ss} holds.
		
		Case (ii).  Let  $\sigma\in S_2$ and
		$$\bar\sigma = (np_1\cdots p_s) \cdots (u_1 q_t\cdots q_1ku_l\cdots u_2)\cdots,$$
		$$\sigma^+ = (n q_1\cdots q_t u_1u_2\cdots u_l kp_1\cdots p_s) \cdots,$$
		and
		$$\sigma^+ = (nu_l\cdots {u_1} q_t\cdots q_1kp_1\cdots p_s)  \cdots.$$
		It follows from Lemma~\ref{Lem:<s>2+=2} that
		\begin{eqnarray}
			&&\langle u_1\cdots u_l k q_1\cdots q_t u_1\rangle + \langle u_1 q_t\cdots q_1ku_l\cdots u_1 \rangle \nonumber \\
			&=& \langle k q_1\cdots q_t u_1u_2\cdots u_l k\rangle + \langle ku_l\cdots {u_1}q_t\cdots q_1k\rangle. \label{equ:kforward}
		\end{eqnarray}
		{Combining} this with the first case derives \eqref{equ:ss=ss}.
		
		The union of all $\sigma^+$ and $\bar\sigma^+$ defined above is equal to $S_3$. In fact, for any permutation $\sigma \in S_3$, if $k\ge p_i$ for $i=1,\dots,s$, then it is corresponding to a permutation in the first case. Otherwise, it is corresponding to a permutation in the second case.
		
		{Together}, we have
		\begin{eqnarray*}
			\Cdet(A_{n\cdot}(a_{k\cdot}))&=& \sum_{\sigma\in S_1} s(\sigma)\langle \sigma \rangle + \sum_{\sigma\in S_2} s(\sigma)\langle \sigma \rangle +\sum_{\sigma\in S_3} s(\sigma)\langle \sigma \rangle\\
			&=&-\sum_{\sigma\in S_3} s(\sigma)\langle \sigma \rangle +\sum_{\sigma\in S_3} s(\sigma)\langle \sigma \rangle\\
			&=& 0.
		\end{eqnarray*}
		
		Furthermore, we can verify that $\Cdet(A_{n\cdot}(\alpha a_{k\cdot})) = \alpha \Cdet(A_{n\cdot}( a_{k\cdot})) =0$.
		This completes the proof.
	\end{proof}

	\begin{Lem}\label{Lem:An_col_ai=0}
		Let $A=(a_{ij})$ be a dual quaternion Hermitian matrix in $\dqset^{n\times n}$.   Then  we have   $\Cdet(A_{\cdot n}(a_{\cdot k}\alpha))=0$ for all $k=1,\dots,n$ and $\alpha\in\dqset$.
	\end{Lem}
	{\begin{proof}
			Consider the three {cases} $S_1$, $S_2$ {and} $S_3$   in {the proof of} Lemma~\ref{Lem:An_row_ai=0}. We {now} prove  that for any permutation  $\sigma \in S_1 \cup S_2$, there is a permutation  $\bar\sigma$ in the same  {case}  and permutations  $\sigma^+, \bar\sigma^+\in S_3$ such that \eqref{equ:ss=ss} holds.
			
			Case (i). Let $\sigma\in S_1$ and
			$$\bar\sigma = (np_1\cdots p_s) \cdots (k q_t\cdots q_1)\cdots,$$
			$$\sigma^+ = (n p_1\cdots p_skq_1\cdots q_t) \cdots, \ \ \bar\sigma^+ = (n p_1\cdots p_skq_t\cdots q_1)  \cdots.$$
			Then we have
			\begin{eqnarray*}
				\langle \sigma \rangle + \langle  \bar\sigma \rangle &=&  \langle np_1\cdots p_s n\rangle \cdots  \langle k q_1\cdots q_t k\rangle \cdots + \langle np_1\cdots p_s n\rangle \cdots  \langle k q_t\cdots q_1 k\rangle \\
				&=&  \langle np_1\cdots p_s k\rangle \alpha \cdots  \langle kq_1\cdots q_t k\rangle \cdots + \langle np_1\cdots p_s k\rangle \alpha \cdots  \langle k q_t\cdots q_1 k\rangle \\
				&=&  \langle np_1\cdots p_sk q_1\cdots q_t k\rangle \alpha \cdots   + \langle np_1\cdots p_sk q_t\cdots q_1 k\rangle \alpha\cdots  \\
				&=&  \langle np_1\cdots p_sk q_1\cdots q_t n\rangle\cdots   + \langle np_1\cdots p_sk q_t\cdots q_1 n\rangle\cdots  \\
				&=& \langle \sigma^+\rangle + \langle  \bar\sigma^+ \rangle ,
			\end{eqnarray*}
			where the second and the fourth equalities follow from the fact that  the last column of $A$ is replaced by  the $k$-th row multiplied by $\alpha$, and the third equality follows from  the fact that  $ \langle k q_1\cdots q_t k\rangle+ \langle k q_t\cdots q_1 k\rangle$ is a dual number and is communicative with any dual quaternion numbers.
			Furthermore, suppose $\sigma$ has $r$ distinct cycles.   Then $s(\sigma)=s(\bar\sigma)=(-1)^{n-r}$ and $s(\sigma^+)=s(\bar\sigma^+)=(-1)^{n-r+1}$. Thus, \eqref{equ:ss=ss} holds.

			Case (ii).  Let  $\sigma\in S_2$ and
			$$\bar\sigma = (np_1\cdots p_s) \cdots (u_1 q_t\cdots q_1ku_l\cdots u_2)\cdots,$$
			$$\sigma^+ = (n p_1\cdots p_s k q_1\cdots q_t u_1\cdots u_l n) \cdots,$$
			and
			$$\sigma^+ = (n p_1\cdots p_s k  u_l\cdots u_1 q_t\cdots q_1 n)  \cdots.$$
			Combining \eqref{equ:kforward}   with the first case derives \eqref{equ:ss=ss}.
			
			The rest of the proof is similar to  that of Lemma~\ref{Lem:An_row_ai=0} and we omit the details here.
			
			This completes the proof.
		\end{proof}
	}
	\begin{Lem}\label{Lem:add_Cdet}
		Let $A=(a_{ij})$ be a dual quaternion Hermitian matrix in $\dqset^{n\times n}$, and $P_{ij;\alpha}$ be an addition matrix.    {Then}    $\Cdet(P_{ij;\alpha}^*AP_{ij;\alpha}) = \Cdet(A)$.
	\end{Lem}
	\begin{proof}
		Without loss of generality, we suppose $i<j$.
		Since $P_{ij;\alpha}=P_{jn}P_{in;\alpha}P_{jn}$,   we only need to show $\Cdet(P_{in;\alpha}^*AP_{in;\alpha}) = \Cdet(A)$. By direct computation, we have
		\begin{eqnarray*}
			&&\Cdet\left(P_{in;\alpha}^*AP_{in;\alpha}\right)  \nonumber\\
			&=&\Cdet\left(A\right) +\Cdet\left(A_{n\cdot}(\alpha^* a_{i\cdot})\right) +\Cdet\left(A_{\cdot n}(a_{\cdot i}\alpha)\right) +\Cdet\left(B(\alpha)\right),
		\end{eqnarray*}
		where $B(\alpha)=(b_{kl}(\alpha))$ satisfies
		{\[b_{kl}(\alpha)= \left\{\begin{array}{ll}
				\alpha^* a_{il}, & \text{ if } k=n \text{ and } l\neq n; \\
				a_{ki}\alpha, & \text{ if } l=n \text{ and } k\neq n; \\
				\alpha^* a_{ii}\alpha, & \text{ if } k=l=n \\
				a_{kl}, & \text{ otherwise. }
			\end{array}\right.\]}
		In {the} other words, $B(\alpha) = P_{n;\alpha}^*B(1)P_{n;\alpha}$.
		Thus, it follows from  Lemmas \ref{Lem:rowprod} and \ref{Lem:An_row_ai=0} that $\Cdet\left(B(\alpha)\right) = \alpha^*\alpha \Cdet(B(1)) = 0$.
		By  Lemmas \ref{Lem:An_row_ai=0} and \ref{Lem:An_col_ai=0}, we have $\Cdet\left(A_{n\cdot}(\alpha^* a_{i\cdot})\right) =\Cdet\left(A_{\cdot n}(a_{\cdot i}\alpha)\right)=0$.
		{Therefore}, we have $\Cdet\left(P_{in;\alpha}^*AP_{in;\alpha}\right) =\Cdet(A)$.
		This completes the proof.
	\end{proof}

	{Recently, Wang et al. \cite{WLWXZ24} proposed   the dual quaternion LU decomposition and  the partial pivoting dual quaternion LU	decomposition.
		Specifically, given a dual quaternion matrix  $A\in\dqset^{n\times n}$, there is a permutation matrix $P$,  a unit lower triangular dual quaternion matrix $L\in\dqset^{n\times n}$,
		and   an upper triangular dual quaternion matrix  $U\in\dqset^{n\times n}$ such that
		\[PA=LU.\]
		Based on this formulation, we have the following lemma.
		\begin{Lem}\label{Lem:u_prod}
			Suppose $A\in\dqset^{n\times n}$ is a dual  {quaternion} unitary matrix.   Then there exists a series of switching matrices  and addition matrices $P_1,\dots,P_t$ such that
			$P_t\cdots P_1A=\mathrm{diag}(d_1,\dots,d_n)$, where $d_1,\dots,d_n\in\dqset$  and  $|d_1d_2\cdots d_n|=1$.
		\end{Lem}
		\begin{proof}
			Let $PA=LU$ be the partial pivoting dual quaternion LU	decomposition of $A$.
			Since $P$, $A$, and $L$ are all invertiable, $U$ is also invertiable.
			Let $U = U_0 D$, where $D=\mathrm{diag}(d_1,\dots,d_n)$ {and $d_i=u_{ii}$ for  $i=1,\dots,n$.}
			Then we have $$U_0^{-1}L^{-1}PA=D=\mathrm{diag}(d_1,\dots,d_n).$$
			Here, $U_0^{-1}$ and $L^{-1}$ are unit upper and lower triangular dual quaternion matrices, respectively. Thus, they are products of addition matrices.

			Furthermore, by the quasi-determinant defined in Ling and Qi \cite{LQ24}, we have {$\qdet(U_0^{-1})=1$, $\qdet(L^{-1})=1$, $\qdet(P)=1$, $\qdet(A^*A)=\qdet(A^*)\qdet(A)=\qdet(A)^2=1$, and}
			\[\qdet(U_0^{-1}L^{-1}PA)=\qdet(U_0^{-1})\qdet(L^{-1})\qdet(P)\qdet(A)=1.\]
			{Therefore}, $\qdet(D) = |d_1d_2\cdots d_n|^2$ $=1$.
			This completes the proof.
	\end{proof}}

	Based on the above lemmas, we are ready to show our main results.
	\begin{Thm}\label{Thm:UAU=A}
		{For any} dual quaternion Hermitian  matrix $A\in\dqset^{n\times n}$ and any {dual quaternion} unitary matrix $U\in\dqset^{n\times n}$,  {we have} $\Cdet(U^*AU)=\Cdet(A)$. 	
	\end{Thm}
	\begin{proof}
		By Lemma~\ref{Lem:u_prod}, there are switching matrices  and addition matrices $P_1,\dots,$ $P_t$ such that
		$U=P_1^{-1}\cdots P_t^{-1}D$,   $D=\mathrm{diag}(d_1,\dots,d_n)$, {and $|d_1\cdots d_n|=1$.} {Since  $P_{ij}^{-1} = P_{ij}$,  $P_{ij;\alpha}^{-1}=P_{ij;-\alpha}$,    {$P_1^{-1},\dots,$} $P_t^{-1}$ are also switching or  addition matrices.}
		Thus, we have
		\begin{eqnarray*}
			\Cdet(U^*AU)&=& \Cdet(D^*(P_t^{-1})^*\cdots (P_1^{-1})^*AP_1^{-1}\cdots P_t^{-1}D)=\Cdet(A).
		\end{eqnarray*}
		{Here}, the last equality follows directly from  Proposition~\ref{prop_Cdet0}.
		This completes the proof.
	\end{proof}
	
	\begin{Thm}\label{thm:cdet_eig}
		Given a   dual quaternion Hermitian  matrix $A\in\dqset^{n\times n}$ and any {dual quaternion} unitary matrix $U\in\dqset^{n\times n}$, {we have}
		\[\Cdet(A) = \prod_{i=1}^n\lambda_i\]
		and
		\[\Cdet(A^*A)=\Cdet(A)^2=\prod_{i=1}^n\lambda_i^2,\]
		where $\lambda_1,\dots, \lambda_n$ are eigenvalues of $A$.
	\end{Thm}
	\begin{proof}
		By~\cite{QL23}, there exists a dual quaternion unitary matrix $U$ such that $UAU^*=\Lambda = \mathrm{diag}(\lambda_1,\dots,\lambda_n)$ and all eigenvalues $\lambda_i$, $i=1,\dots,n$ are dual numbers.
		Combining this with Theorem~\ref{Thm:UAU=A}, we have $\Cdet(A) = \Cdet(U^*AU)=\Cdet(\Lambda)=\lambda_n\cdots\lambda_1$.
		Since  dual numbers are communicative, the first conclusion is derived.
		
		The second conclusion  follows from {$\Lambda^*=\Lambda$,} $A^*A = U^*\Lambda^2U$ and  Theorem~\ref{Thm:UAU=A}.
		
		This completes the proof.		
	\end{proof}
	
	\section{The  Moore Determinant and the Eigenvalues}
	
	\begin{Thm}\label{Thm:Mdet=Cdet}
		Let $A=(a_{ij})$ be  {a dual quaternion} Hermitian matrix in $\dqset^{n\times n}$.
		Then we have
		\begin{equation}\label{equ:Cdet2}
			\Mdet(A) = \Cdet(A).
		\end{equation}
		Furthermore,    {we have}
		$$\Mdet(A)= \prod_{i=1}^n\lambda_i,$$
		where $\lambda_1,\dots, \lambda_n$ are the eigenvalues of $A$.
	\end{Thm}
	\begin{proof}
		Let $\sigma=\sigma_1\cdots \sigma_r \in S_n^M$ be a  permutation in the Moore determinant.
		For each $i=1,\dots,r$,   denote $$\sigma_i = (p_{i1}\cdots p_{is_i} k_i q_{i1} \cdots q_{it_i}),$$
		where $p_{i1}$ and $k_i$ are the minimal and maximal integers in $\sigma_i$, respectively.  Denote
		$\bar\sigma_i = (p_{i1} q_{it_i}\cdots q_{i1} k_i p_{is_i} \cdots p_{i2})$, $\sigma_i^+ = (k_i q_{i1} \cdots q_{it_i}p_{i1}\cdots p_{is_i})$, and  $\bar\sigma_i^+ = (k_i p_{is_i} \cdots p_{i1}q_{it_i} \cdots q_{i1} ).$
		By Lemma~\ref{Lem:<s>2+=2}, we have $$\langle \sigma_i \rangle + \langle \bar \sigma_i \rangle = \langle \sigma_i^+ \rangle + \langle \bar \sigma_i ^+\rangle\in\drset.$$
		
		Thus, we have
		\begin{eqnarray*}
			&&\left(\langle \sigma_1 \rangle + \langle \bar \sigma_1\rangle \right) \cdots \left(\langle \sigma_r \rangle + \langle \bar \sigma_r \rangle \right) \\
			&= & \left(\langle \sigma_1 ^+\rangle + \langle \bar \sigma_1^+ \rangle \right) \cdots \left(\langle \sigma_r^+ \rangle + \langle \bar \sigma_r^+ \rangle \right) \\
			&= & \left(\langle \sigma_{i_1} ^+\rangle + \langle \bar \sigma_{i_1}^+ \rangle \right) \cdots \left(\langle \sigma_{i_r}^+ \rangle + \langle \bar \sigma_{i_r}^+ \rangle \right).
		\end{eqnarray*}
		Here, $i_1,\dots, i_r$ are obtained by sorting $k_i$ in the  descending order
		and the second equality holds  since the dual numbers are communicative.
		Let {$s_j \in \{\sigma_j,\bar\sigma_j \}$ for $j=1,\dots,n$.} Then $\sigma = s_1\cdots s_r \in S_n^M$ is a permutation in the Moore determinant.
		Similarly, 		let {$s_j^+\in \{\sigma_{i_j}^+,\bar\sigma_{i_j}^+ \}$ for $j=1,\dots,n$.} Then $\sigma^+ = s_1^+\cdots s_r^+ \in S_n^C$ is a permutation in the Chen determinant.
		{Furthermore, those cycle factors are obtained by  reversing   or switching the original cycle factor. Reversing a cycle factor derives  the inverse or conjugate of the corresponding permutation,  which does not change the sign. Switching the cycle factors does not change the corresponding  permutation. Thus,   the    signs of all permutations above  are the same. By}
		going through all permutations in the Moore determinant, we have $	\Mdet(A) = \Cdet(A)$.
		
		The second conclusion of this theorem holds directly from {equation \eqref{equ:Cdet2}} and Theorem~\ref{thm:cdet_eig}.
		This completes the proof.
	\end{proof}
	{Kyrchei \cite{Kyr08, Kyr12} showed that the Kyrchei row and column determinants of a quaternion Hermitian matrix are equal to the Moore determinant. We could also extend this result to dual quaternion Hermitian matrices and show the Kyrchei row and column determinants of a quaternion Hermitian matrix are equal to the Moore determinants, which are also equal to the products of the eigenvalues.}

	\begin{Thm}
		Suppose $A\in\dqset^{n\times n}$ is   a dual quaternion Hermitian matrix and $\lambda_1,\dots,\lambda_n\in\drset$ are its eigenvalues. Then $\Mdet(A) =0$ if and only if $A_s$ is singular and $A$ has  {at least} one zero or two infinitesimal eigenvalues; namely, there is $i,j\in \{1,\dots,n\}$ such that  $\lambda_i=0$ or  $\lambda_{i,s}=\lambda_{j,s}=0$.
	\end{Thm}	
	\begin{proof}
		By Theorem~\ref{Thm:Mdet=Cdet}, we have
		\[\Mdet(A)= \prod_{i=1}^n\lambda_i=\prod_{i=1}^n\lambda_{is} + \sum_{j=1}^n \prod_{i=1,i\neq j}^n\lambda_{is}\lambda_{jd}\epsilon.\]

		{On one hand, if there is $i,j\in \{1,\dots,n\}$ such that  $\lambda_i=0$ or  $\lambda_{i,s}=\lambda_{j,s}=0$, we can check that $\Mdet(A)=0$ directly.
			On the other hand, if  $\Mdet(A)=0$, then there is $j\in\{1,\dots,n\}$ such that  $\lambda_{js}=0$ and $\prod_{i=1,i\neq j}^n\lambda_{is}\lambda_{jd}=0$.
			Thus, either $\lambda_{jd}=0$ or there is $i\neq j$ such that $\lambda_{is}=0$.}
		This completes the proof.
	\end{proof}
	
	\section{The Characteristic Polynomial}
	
	Suppose $A\in\dqset^{n\times n}$ is   a dual quaternion Hermitian matrix and $\lambda\in \dqset$ is an arbitrary dual number.
	Define
	\begin{equation}\label{def:charpoly}
		p(\lambda) = \Mdet(\lambda I-A).
	\end{equation}
	{Assume that $\lambda_1,\dots,\lambda_n$ are eigenvalues of $A$. Then \eqref{def:charpoly} is equivalent to
		\begin{equation}\label{def:charpoly_eigs}
			p(\lambda) = (\lambda-\lambda_1)(\lambda-\lambda_2)\cdots(\lambda-\lambda_n).
		\end{equation}
		The later definition is also used in Qi and Cui \cite{QC24}.  It was shown in \cite{QC24} that all eigenvalues of $A$ are also roots of the characteristic polynomial. 
		{The} reverse does not hold {in general}.
		In fact, a dual number $\lambda$ is a root of $p(\lambda)$ either if $\lambda$ is an eigenvalue of $A$ or $\lambda_s$ is a multiple eigenvalue of $A_s$. {However, the reverse holds true when all eigenvalues of $A_s$ are single.}
		
		\begin{Cor}
			Suppose all eigenvalues of $A_s$ are single. Then a dual number is the root  of the  characteristic polynomial \eqref{def:charpoly} if and only if it is an  eigenvalue  of $A$.
		\end{Cor}
		
		Very recently, Ling and Qi \cite{LQ24} proposed a characteristic polynomial by the quasi-determinant of dual quaternion matrices. In their definition, there is
		\begin{equation}\label{def:charpoly_LQ24}
			p_q(\lambda)= \mathrm{det}_q(\lambda I - A) = |\lambda-\lambda_1|^2|\lambda-\lambda_2|^2\cdots |\lambda-\lambda_n|^2.
		\end{equation}
		We may verify that
		$p_q(\lambda) = |p(\lambda)|^2.$
		{In other words, the sign of the determinant is lost in the quasi-determinant. Furthermore, it}
		was shown in \cite{LQ24} that all eigenvalues of $A$ are also roots of the characteristic polynomial \eqref{def:charpoly_LQ24}. However, the reverse does not hold, even if all eigenvalues of  $A_s$ are single.
		
		\medskip
		
		{{\bf Acknowledgment}
			We are grateful to Prof. Zhigang Jia for helpful discussions.}

		\bigskip


	\end{document}